\documentclass[11pt,reqno]{amsart}
\usepackage{color}

\usepackage{amsmath}
\usepackage{amsfonts,amscd}
\usepackage{amssymb}
\usepackage{url}
\usepackage{hyperref}

\usepackage[english]{babel}
\usepackage{booktabs}
\usepackage{tikz}

\theoremstyle{plain}
\newtheorem{theorem}                 {Theorem}      [section]

\newtheorem{proposition}  [theorem]  {Proposition}

\theoremstyle{definition}

\newtheorem{remark}       [theorem]  {Remark}

\numberwithin{equation}{section}

\def \rn{{\mathbb R}}

\def \F{\mathcal F}
\def \H{\mathcal H}

\def \V{\mathcal V}

\def \nab#1#2{\nabla_{#1}{#2}}

\def \lb#1#2{[#1,#2]}

\def \veps#1{\varepsilon_#1}

\def \g{\mathfrak{g}}

\def \k{\mathfrak{k}}

\def \m{\mathfrak{m}}

\def \p{\mathfrak{p}}

\DeclareMathOperator{\ad}{ad}

\def \SLR#1{\text{\bf SL}_{#1}(\rn)}
\def \SL2{\widetilde{\text{\bf SL}}_{2}(\rn)}

\def \SO#1{\text{\bf SO}(#1)}

\def \SU#1{\text{\bf SU}(#1)}
\def \su#1{\mathfrak{su}(#1)}

\DeclareMathOperator{\trace}{trace}

\numberwithin{equation}{section}

\allowdisplaybreaks

\begin{document}

	
	
\author{Sigmundur Gudmundsson}
\address{Mathematics, Faculty of Science\\
University of Lund\\
Box 118, Lund 221 00\\
Sweden}
\email{Sigmundur.Gudmundsson@math.lu.se}
	
\author{Thomas Jack Munn}
\address{Mathematics, Faculty of Science\\
University of Lund\\
Box 118, Lund 221 00\\
Sweden}
\email{Thomas.Munn@math.lu.se}

\title
[Harmonic Morphisms and Minimal Conformal Foliations on Lie Groups]
{Harmonic Morphisms on Lie groups\\ and Minimal Conformal Foliations\\ of Codimension Two}

\begin{abstract}
Let $G$ be a Lie group equipped with a left-invariant semi-Riemannian metric.  Let $K$ be a semisimple subgroup of $G$ generating a left-invariant conformal foliation $\F$ of codimension two on $G$.  We then show that the foliation $\F$ is minimal.  This means that locally the leaves of $\F$ are fibres of a complex-valued harmonic morphism. In the Riemannian case, we prove that if the metric restricted to $K$ is biinvariant then $\F$ is totally geodesic.
\end{abstract}
	
	
\maketitle

\section{Introduction}
\label{section-introduction}	

In differential geometry the study of minimal submanifolds, of a given Riemannian ambient manifold, plays a central role.  In the literature these submanifolds are often either surfaces of dimension two or submanifolds of codimension one.  In this work we are interested in minimal submanifolds of codimension {\it two}. 

Submersive harmonic morphisms $\phi:(M,g)\to (N^2,h)$ from a Riemannian manifold with values in a surface are useful tools for the construction of such objects.  This is due to the following result by P. Baird and J. Eells from 1981, see \cite{Bai-Eel}. This has later been generalised to the semi-Riemannian situation, see \cite{Gha-Gud-Ott-1}.  For the theory of semi-Riemannian manifolds we refer the reader to the exellent book \cite{ONe} by B. O'Neill.

\begin{theorem}\cite{Bai-Eel}
Let $\phi:(M,g) \to (N^2,h)$ be a horizontally conformal submersion from a Riemannian manifold to a surface. Then $\phi$ is harmonic if and only if $\phi$ has minimal fibres.
\end{theorem}

Harmonic morphisms are the solutions to an  over-determined,  non-linear system of partial differential equations.  For this reason they are difficult to find, and there exist 3-dimensional Riemannian Lie groups which do not allow for any solutions, not even locally, see \cite {Gud-Sve-5}.  For the general theory of harmonic morphisms between Riemannian
manifolds we refer to the excellent book \cite{Bai-Woo-book} and the regularly updated online bibliography \cite{Gud-bib}.

The next statement was proven by J.C. Wood in his paper \cite{Wood1986} from 1986.  This provides an interesting relation to the theory of conformal foliations with minimal leaves of codimension two.

\begin{theorem}\cite{Wood1986}
A foliation of codimension two on a Riemannian manifold produces harmonic morphisms if and only if it is conformal and has minimal leaves.
\end{theorem}

In this work we are interested in Lie groups $(G,g)$ equipped with a left-invariant semi-Riemannian metric $g$.  We further assume that $K$ is a Lie subgroup of $G$ generating a left-invariant foliation $\F$ on $G$, of codimension two.  For the special situation when the subgroup $K$ is semisimple we prove the following two results.

\begin{theorem}\label{theorem-totally-geodesic}
Let $G$ be a Lie group with a semisimple subgroup $K$  generating a left-invariant foliation $\F$ on $G$, of codimension two.  Let $g$ be a left-invariant semi-Riemannian metric on $G$ such that the restriction $h$ of $g$ to $K$ is induced by the negative Killing form of $K$.  If the foliation $\F$ on $G$ is conformal, then it is totally geodesic.
\end{theorem}

\begin{theorem}\label{theorem-minimal}
Let $(G,g)$ be a Lie group equipped with a left-invariant semi-Riemannian metric.  Further let $K$ be a subgroup generating a left-invariant conformal foliation $\F$ on $G$, of codimension two.  If $K$ is semisimple then the foliation $\F$ is minimal.
\end{theorem}

Our study is motivated by the investigation performed in \cite{Gha-Gud-Tur-1}, where the following four interesting cases are examined.
\vskip .2cm
\begin{center}
\begin{tabular}{c|c|c}
$K$ & \text{compact} & \text{non-compact} \\
\hline
\text{semisimple}  & $\SU 2\times\SU 2$ & $\SU 2\times\SLR 2$ \\
\text{non-semisimple}  & $\SU 2\times\SO 2$ & $\SLR 2\times\SO 2$ \\
\end{tabular}
\end{center}
\vskip .1cm
This led the authors to conjectures related to our results in Theorems \ref{theorem-totally-geodesic} and \ref{theorem-minimal}.


\section{Conformal Foliations of Codimension Two}
\label{section-conformal-foliations}

Let $(M,g)$ be a semi-Riemannian manifold, $\V$ be an integrable distribution on $M$ and $\H$ its orthogonal complementary distribution.  We will also denote by $\V,\H$ the orthogonal projections onto the corresponding subbundles of $TM$ and by $\F$ the foliation tangent to $\V$. Then the second fundamental form $B^\V$ of $\V$ is given by $$B^{\V}(V,W) = \tfrac{1}{2}\cdot\H(\nab VW + \nab WV), \ \ \text{ where } V,W\in\V.$$
The corresponding second fundamental form $B^\H$ of $\H$ satisfies
$$B^{\H}(X,Y) = \tfrac{1}{2} \cdot \V(\nab XY + \nab YX), \ \ \text{ where } X,Y \in \H.$$
		
The foliation $\F$ tangent to $\V$ is said to be {\it conformal} if there exists a vector field $V\in\V$ such that
$$B^{\H}=g\otimes V$$
and $\F$ is said to be {\it semi-Riemannian} if $V =0$. Furthermore, $\F$ is {\it minimal} if $\trace B^{\V} =0$ and {\it totally geodesic} if $B^{\V}=0$. This is equivalent to the leaves of $\F$ being minimal and totally geodesic submanifolds of $M$, respectively.
\smallskip

In the following we denote by $\veps X$ and $\veps Y$ the causality of the horizontal vectors $X,Y\in\H$, respectively.

\begin{proposition}\label{prop-general-conformality}
\cite{Gha-Gud-Ott-1}
Let $(M,g)$ be a semi-Riemannian manifold, $\V$ be an integrable distribution on $M$ of codimension two and $\H$ be its orthogonal complementary distribution. Then the foliation $\F$ tangent to $\V$ is conformal if and only if 
$$\veps X\cdot B^\H(X,X)-\veps Y\cdot B^\H(Y,Y)=0\ \ \text{and}\ \ B^\H(X,Y)=0,$$
for any local orthonormal frame $\{X,Y\}$ for the horizontal distribution $\H$.  If $\F$ is conformal then it is semi-Riemannian if and only if 
$$\veps X\cdot B^\H(X,X)+\veps Y\cdot B^\H(Y,Y)=0.$$
\end{proposition}

\section{Lie Foliations of Codimension Two}
\label{section-Lie-foliations}

Let $(G,g)$ be an $(n+2)$-dimensional semi-Riemannian Lie group i.e. a Lie group $G$ equipped with a left-invariant semi-Riemannian metric $g$.  Further we assume that $K$ is a subgroup of $G$, of  codimension two, equipped with the induced metric.
\smallskip

Let $\{V_1,\dots ,V_n,X,Y\}$ be an orthonormal basis for the Lie algebra $\g$, such that $X,Y\in\H$.  Then the Koszul formula implies that the second fundamental forms $B^{\V}$ and $B^{\H}$ of the horizontal and vertical distributions  satisfy
\begin{eqnarray*}
2\cdot B^{\V}(V_j,V_k) 
&=&\veps X\cdot\big( (g([X,V_j],V_k)+g([X,V_k],V_j)\big)\cdot X\\
&&+\,\veps Y\cdot \big( g([Y,V_j],V_k)+g([Y,V_k],V_j)\big)\cdot Y,
\end{eqnarray*}
\begin{equation}\label{equation-BH}
B^{\H}(X,Y) 
=\sum_{i=1}^n \veps{i} \cdot \big( g([X,V_i],Y)+g([Y,V_i],X)\big)\cdot V_i.
\end{equation}

The following results, first proven in \cite{Gha-Gud-Ott-1}, allow us to describe the structure of conformal foliations on Lie groups. Both are vital to the proof of our results later on.
	
\begin{proposition}\label{proposition-H[[V,V],H]=0}
\cite{Gha-Gud-Ott-1}
Let $(G,g)$ be a semi-Riemannian Lie group with a  subgroup $K$ generating a left-invariant conformal foliation $\F$ on $G$. Let $\V$ be the integrable distribution tangent to $\F$ and $\H$ be the orthogonal complementary distribution of dimension two. Then
$$\H [[\V,\V],\H] =0.$$
\end{proposition}

In the case when the subgroup $K$ of $G$ is semisimple  we have the following immediate consequence of  Proposition \ref{proposition-H[[V,V],H]=0}.

\begin{proposition}\label{proposition-foliation-is-Riemannian}
\cite{Gha-Gud-Ott-1}
Let $(G,g)$ be a semi-Riemannian Lie group with a semisimple subgroup $K$, of codimension two, generating a left-invariant conformal foliation $\F$ on $G$. Then the foliation $\F$ is semi-Riemannian.
\end{proposition}

\begin{proof}
Let $\V$ be the integrable distribution tangent to $\F$ and $\H$ be the orthogonal complementary distribution of dimension two.  Since the subgroup $K$ is semisimple we know that $[\V,\V]=\V$ and it then follows from Proposition \ref{proposition-H[[V,V],H]=0} and equation (\ref{equation-BH}) that $\H [\V,\H]=0$.  This proves the statement.
\end{proof}

Another interesting result, concerning the structure of the Lie groups involved, is the following.  This leads to important simplifications by allowing us to consider the semisimple subgroup K in terms of its constituent simple factors.

\begin{proposition}\label{proposition-Gha-Gud-Ott-Product}
\cite{Gha-Gud-Ott-1}
Let $(G,g)$ be a semi-Riemannian Lie group with a semisimple product subgroup $K=K_1\times \cdots \times K_r$ generating a left-invariant conformal foliation on $G$. Let $\mathfrak{g} = \mathfrak{k}_1\oplus \cdots \oplus \mathfrak{k}_r \oplus \mathfrak{m}$ be an orthogonal decomposition of the Lie algebra $\mathfrak{g}$ of $G$ and $\mathfrak{m}$ be two dimensional. For each $k$, let $\V_k$ be the integrable distribution generated by $\k_k$, $\V = \V_1 \oplus \cdots \oplus\V_r$ and $\H$ be its orthogonal complementary distribution generated by $\mathfrak{m}$. Then for all $j \neq k$ we have
$$ \V_j [\V_k,\H] = 0 .$$
\end{proposition}

\section{Minimality and Total Geodecity}
	
Let $(G,g)$ be an $(n+2)$-dimensional semi-Riemannian Lie group with a {\it semisimple} subgroup $K$, generating a left-invariant conformal foliation $\F$ on $G$, of codimension two. Further let $\g=\k\oplus \m$ be an orthogonal decomposition of the Lie algebra $\g$ of $G$, where $\k$ is the Lie algebra of $K$. Let $\{ V_1,\dots,V_n,X,Y\}$ be an orthonormal basis for $\g$ such that $V_1,\dots,V_n$ generate the subalgebra $\k$. Then we can describe the Lie brackets of $\k$ by the following system
\begin{eqnarray*}
[V_i,V_j]&=&\sum_{k=1}^n c^k_{ij} V_k,
\end{eqnarray*}
where the $c^k_{ij}$ are real numbers.
The subgroup $K$ is semisimple so $[\k,\k ]=\k$.  Hence it follows from Proposition \ref{proposition-H[[V,V],H]=0} that $\H[\V,\H]=0$, so we have that
$$
\ad_X(V_i)=[X,V_i ]=\sum_{k=1}^n x^k_i V_k,\ \ 
\ad_Y(V_i)=[Y,V_i ]=\sum_{k=1}^n y^k_i V_k,
$$
and $X$ and $Y$ can be chosen such that 
$$[X,Y]=\rho X+ \sum_{k=1}^n \theta^k V_k,$$
for some real coefficients $x^k_i,y^k_i,\rho,\theta^k$.

The structure constants of $\g$ can be used to describe when the foliation $\F$ is minimal or even totally geodesic. First we notice that
\begin{eqnarray*}
2\cdot B^{\mathcal{V}}(V_i,V_j) 
&=& \veps X\cdot \big(g([X,V_i],V_j) +g([X,V_j],V_i)\big)\cdot X\\
&&+\, \veps Y\cdot \big( g([Y,V_i],V_j)+g([Y,V_j],V_i)\big)\cdot Y\\
&=& \veps X\cdot\sum_{k=1}^n \big( g( x_i^k V_k ,V_j)+g( x_j^k V_k,V_i)\big)\cdot X\\
&&+\, \veps Y\cdot\sum_{k=1}^n  \big (g(y_i^k V_k,V_j)+g(y_j^k V_k,V_i)\big)\cdot Y\\
&=& \veps X\cdot (\veps j x_i^j + \veps i x_j^i)\cdot X + \veps Y\cdot (\veps j y_i^j + \veps i y_j^i)\cdot Y. 
\end{eqnarray*}
From this we see that the foliation $\F$ is {\it minimal} if and only if
$$ \sum_{i=1}^n \veps i\, x_i^i = 0 = \sum_{i=1}^n \veps i\, y_i^i.$$
Furthermore, $\F$ is {\it totally geodesic} if and only if for all $1\le i,j\le n$ we have 
$$\veps j\, x_i^j+\veps i\, x_j^i=0=\veps j\, y_i^j + \veps i\, y_j^i.$$

\section{The Proofs of our Main Results}

In this section we specify further the left-invariant semi-Riemannian metrics on the Lie groups that we are dealing with.  For the details from Lie theory, we refer to the standard work \cite{Hel} of S. Helgason.
\smallskip

Let $K$ be a real semisimple Lie group with Lie algebra $\k$.  Let $B:\k\times\k\to\rn$ be its symmetric, bilinear and non-degenerate Killing form with  $$B(V,W)=\trace(\ad_V \circ \ad_W).$$
Let $\tilde K$ be the maximal compact subgroup of $K$ with Lie algebra $\tilde\k$.  Then we have an orthogonal decomposition $\k=\tilde\k\oplus\p$ of $\k$ with respect to the Killing form $B$, which is negative definite on $\tilde\k$ and positive definite  on $\p$.
\smallskip

\begin{proof}[Proof of Theorem \ref{theorem-totally-geodesic}]
Let $B:\g\times\g\to\rn$ be the Killing form of the Lie algebra $\g$ of $G$ and suppose that the metric $g$ restricted to $K$ is equal to the negative Killing form i.e.
$$ g|_{\V \times \V} = -B_{\V \times \V}.$$
 Further, let $\{V_1,\dots V_n\}$ be an orthonormal basis for the semisimple subalgebra $\k$ of $\g$ with respect to $g$ and can write $$g(V_i,V_j)=-B(V_i,V_j)=\veps{i}\delta_{ij},$$ for all $1\le i,j\le n$.  Then
\begin{eqnarray*}
B([V_i,V_j],X) 
&=& B(V_i,[V_j,X])\\	
&=&-B(V_i,\sum_{k=1}^n x^k_j V_k)\\	
&=& \veps {i} x_j^i.
\end{eqnarray*}
Since $B([V_i,V_j],X) = -B([V_j,V_i],X)$, the above steps show that 
$$\veps {i}\, x_j^i +\veps {j}\, x_i^j=0.$$ 
By an identical argument replacing $X$ with $Y$ it follows that 
$$\veps {i}\, y_j^i +\veps {j}\, y_i^j=0.$$ 
This implies that $\mathcal{F}$ is totally geodesic. 	
\end{proof}

\begin{remark}
It should be noted that it follows from Theorem \ref{theorem-totally-geodesic} that if the semisimple subgroup $K$ of $G$ is compact and equipped with a biinvariant Riemannian metric, then the foliation $\F$ is totally geodesic.  This is due to the fact that such a metric is always proportional to the one induced by the negative Killing form in the compact case.
\end{remark}

Let $\theta:\k\to\k$ be a {\it Cartan involution} on $\k$, i.e. an involutive automorphism, such that the bilinear form $(V,W)\mapsto -B(V,\theta(W))$ is positive definite on $\k$.  Such an involution is unique up to inner automorphisms, see p.185 of \cite{Hel} for details.  This induces the Cartan-Killing metric $g$ on $K$ satisfying
$$g(V,W) = -B(V,\theta (W)),$$
which is Riemannian.

Let $\{V_1,\dots, V_n\}$ be an orthonormal basis, with respect to the Cartan-Killing metric $g$, then for any $\varepsilon =(\veps  1,\dots,\veps n)\in\{\pm 1\}^n$
we define the semi-Riemannian metric $g_\varepsilon$ on $K$ by
$$g_\varepsilon (V_i,V_j) = \veps i\cdot g(V_i,V_j).$$

The proof of Theorem \ref{theorem-minimal} is rather straightforward, first we will show in Proposition \ref{Proposition-CartanKillingMetricMinimal} that there exists a metric of any given signature such that $\mathcal{F}$ is minimal.  For a fixed signature we then show that the {\it minimality} of the foliation $\F$ is invariant of the left-invariant metric $g_\varepsilon$ chosen.

\begin{proposition} \label{Proposition-CartanKillingMetricMinimal}
Let $K$ be a subgroup of the Lie group $G$ generating a left-invariant conformal foliation $\mathcal{F}$ of $G$ of codimension two. If $K$ is semisimple and equipped with one of the semi-Riemannian metrics $g_\varepsilon$, then $\mathcal{F}$ is minimal.
\end{proposition}

\begin{proof}
Since $K$ is equipped with the semi-Riemannian metric $g_\varepsilon$, it follows that  $B(V_i,V_j) = \veps{i} \delta_{ij}\theta_j$ for an orthonormal basis $\{V_1,\dots V_n\}$.  Here we denote $\theta(V_j)$ by  $\theta_j$.  Then
\begin{eqnarray*}
B([V_i,V_j],X) 
&=& B(V_i,[V_j,X])\\	
&=&-B(V_i,\sum_{k=1}^n x^k_j V_k)\\	
&=& \veps {i} \theta_{i} x_j^i.
\end{eqnarray*}
Since $B([V_i,V_j],X) = -B([V_j,V_i],X)$, the above steps show that 
$$\veps {i} \theta_{i} x_j^i = -\veps {j} \theta_{j} x_i^j.$$ 
Then for $i=j$ we get that $x_i^i=-x_i^i = 0$. By an identical argument replacing $X$ with $Y$ it follows that $y^i_i=-y^i_i=0$. This implies that $\mathcal{F}$ is minimal. 	
\end{proof}

We are now ready to prove our main result.

\begin{proof}[Proof of Theorem \ref{theorem-minimal}]
Let $\{V_1,...,V_n,X,Y\}$ be an orthonormal basis for $\mathfrak{g}$ with respect to the metric $g$ such that $V_1,...,V_n$ generate the subalgebra $\k$. 
Then notice that
$$\trace (\ad_X) = \sum_{i=1}^n \veps i x_i^i +g(\ad_X (Y),Y)$$  and $$\trace (\ad_Y) = \sum_{i=1}^n \veps i y_i^i + g(\ad_Y(X),X).$$
	
Now equip $G$ with an additional left-invariant metric $\hat g$, which we can fully describe by its action on $\mathfrak{g}$.
First define $\hat g|_{\V \times\V}$ to be the left-invariant semi-Riemanian metric $g_\varepsilon$ with the same signature as $\hat g|_{\V \times \V}$ and then let
$$
\hat g|_{\H \times \H} 
= g|_{\H \times \H}, \ \ \hat g|_{\V \times \H} 
= g|_{\V \times \H}.
$$
Then we can use the Gram-Schmidt process to obtain an orthonormal basis $\{W_1,...,W_n,X,Y\}$ with respect to the metric $\hat g$. Since both metrics are left-invariant and of the same signature, changing the metric simply amounts to a change of basis. Without loss of generality we choose an ordering so that $$g(V_i,V_i) = \hat{g}(W_i,W_i) = \veps{i}.$$ Then it follows from Proposition \ref{Proposition-CartanKillingMetricMinimal} that the structure constants with respect to the new metric $\hat g$,  $$\veps{i} \hat x_i^i=\hat g(W_i,[X,W_i]),$$ 
are all equal to zero. Then since the trace of an operator is invariant under the change of basis, and the fact that both $X$ and $Y$ are unchanged by the above Gram-Schmidt process, we get that
\begin{eqnarray*}
\sum_{i=1}^n \veps i x_i^i  
&=& \trace (\ad_X) -g(\ad_X(Y),Y)\\
&=& \trace (\ad_X) -\hat g(\ad_X(Y),Y)\\
&=& \sum_{i=1}^n \veps i \hat x_i^i \\
&=& 0.
\end{eqnarray*}
Similarly we obtain 
$$\sum_{i=1}^n \veps i y_i^i = \sum_{i=1}^n \veps i \hat y_i^i=0,$$ so $\F$ is minimal.
\end{proof}

\section{The Berger metrics on the subgroup $\SU 2$}

If we assume, in Theorem \ref{theorem-totally-geodesic}, that the semisimple subgroup $K$ of $G$ is compact then it is   equipped with a biinvariant metric induced by the negative Killing form of the Lie algebra $\k$.  This implies that the corresponding induced left-invariant foliation $\F$ on $G$ is totally geodesic. In this section we present a $6$-dimensional family of Riemannian Lie groups $(G,g)$ with  the compact subgroup $K=\SU 2$, equipped with the well-known Berger metrics $h_\lambda$ parametrised by the positive real number $\lambda$. These metrics are not biinvariant when $\lambda\neq 1$.
\smallskip

The special unitary group $\SU 2$ is diffeomorphic to the
standard three dimensional sphere $S^3$.  Its Lie algebra is generated by
$$
e_1=\begin{pmatrix}0 & -1\\ 1 &  0\end{pmatrix}, \ \
e_2=\begin{pmatrix}i &  0\\ 0 & -i\end{pmatrix}, \ \
e_3=\begin{pmatrix}0 &  i\\ i &  0\end{pmatrix}.
$$
As a compact Lie group $\SU 2$ carries the well-known
one dimensional family of left-invariant Berger metrics $h_\lambda$ with $\lambda\in\rn^+$. 

For the Berger metric $h_\lambda$ on the special unitary group $\SU 2$ we have the orthonormal basis $\{A=\lambda\, e_1,B=e_2,C=e_3\}$ for its Lie algebra $\su 2$ satisfying the bracket relations
$$\lb AB=2\lambda C,\ \ \lb CA=2\lambda B,\ \ \lb BC=2A/\lambda.$$
We then extend this to an orthonormal basis $\{A,B,C,X,Y\}$ for the Lie algebra $\g$ of a 5-dimensional Lie group $G$, satisfying the additional bracket relations
\begin{eqnarray*}
\lb AX &=& -\lambda^2x_3 B-\lambda^2x_5 C,\\
\lb AY &=& -\lambda^2x_4 B-\lambda^2x_6 C,\\
\lb BX &=& x_3 A+z_3 C,\\
\lb BY &=& x_4 A+z_4 C,\\
\lb CX &=& x_5 A-z_3 B,\\
\lb CY &=& x_6 A-z_4 B,\\
\lb XY &=& \rho X+\theta_1 A+\theta_2 B+\theta_3 C,
\end{eqnarray*}
where $\theta_1,\theta_2,\theta_3$ are given by
\begin{equation*}
\begin{pmatrix}
\theta_1  \\
\theta_2 \\
\theta_3
\end{pmatrix}
=\frac 12
\begin{pmatrix}
\rho\, z_3/\lambda+\lambda\,(x_3\,x_6-x_4\,x_5) \\
\lambda\,(\rho\, x_5-x_3\,z_4+x_4\,z_3) \\
-\lambda\,(\rho\, x_3+z_4\,x_5-z_3\,x_6)
\end{pmatrix}.
\end{equation*}
Here the variables $x_3,x_4,x_5,x_6,\rho\in\rn$ and $\lambda\in\rn^+$ parametrise a six dimensional family of Riemannian Lie groups with subgroup $(K,h_\lambda)$.  According to Theorem 4.1 of \cite{Gud-12} this generates a left-invariant conformal foliation $\F$ with minimal leaves of codimension two.  Furthermore the foliation is totally geodesic if and only if $\lambda=1$ i.e. the Berger metric $h_\lambda$ is biinvariant.

\section{The Riemannian Leaf Space $L^2=G/K$ }

Let $(G,g)$ be a connected Lie group equipped with a left-invariant Riemannian metric and let $K$ be a {\it closed} semisimple subgroup of $G$ generating a conformal foliation $\F$ on $G$.  Since $K$ is closed there exists a unique differentiable structure on the leaf space $L^2=G/K$ such that the natural projection $\pi:G\to G/K$ is differentiable.  

According to Proposition \ref{proposition-foliation-is-Riemannian}, 
the conformality of the foliation $\F$ implies that it is Riemannian.  This means that for all $V\in\V$ and $X,Y\in\H$ we have
$$g([V,X],Y)+g(X,[V,Y])=0.$$
In other words, the metric $g$ on $G$ is $\text{ad}(K)$-invariant and hence induces a unique Riemannian metric $h$ on the $2$-dimensional leaf space $L=G/K$ such that the projection $\pi:(G,g)\to (L^2,h)$ is a Riemannian submersion.  Since the $2$-dimensional
Riemannian space $G/K$ is homogeneous it is of constant Gaussian curvature $K_L$.  For this special situation, we will now employ O'Neill's famous curvature formula 
\begin{equation}\label{equation-ONeill}
K_L=K(X,Y)+\tfrac 34\cdot |\V [X,Y]|^2
\end{equation}
for Riemannian submersions, see \cite{ONe-66}, to determine the value of $K_L$.  Here $K(X,Y)$ is the sectional curvature of the two dimensional horizontal space $\H$ generated by $X,Y\in\g$.

The following useful result can be found in Milnor's paper \cite{Mil} with a incomplete proof sketch.  For this reason we have produced a full proof, to be found in Appendix \ref{appendix}.

\begin{proposition}
Let $(G,g)$ be a Lie group equipped with a left-invariant Riemannian metric and 
$\{e_1,e_2,\dots ,e_m\}$ be an orthonormal basis for the Lie algebra $\g$ of $G$.  If the Lie brackets of $\g$ are given by 
$$[e_i,e_j]=\sum_{k=1}^m\lambda_{ij}^k\,e_k$$
then the sectional curvatures $K(e_i,e_j)$ of $(G,g)$ satisfy
\begin{eqnarray}\label{general-sectional-sg}
K(e_i,e_j)
&=& -\sum_{k=1}^m \lambda_{kj}^j \lambda_{ki}^i
	-\tfrac{1}{2}\cdot\sum_{k=1}^m\lambda_{ij}^k (\lambda_{ij}^k+ \lambda_{ik}^j +\lambda_{kj}^i) \\ 
&& +\tfrac{1}{4}\sum_{k=1}^m (\lambda_{ij}^k + \lambda_{ki}^j + \lambda_{kj}^i) \cdot ( \lambda_{ji}^k +\lambda_{ki}^j+\lambda_{kj}^i).\notag
\end{eqnarray}
\end{proposition}

\begin{proof}
The detailed calculations can be found in Appendix \ref{appendix}.
\end{proof}

For the horizontal distribution $\H$ of the above mentioned foliation $\F$ on the Lie group $G$, we can now choose an orthonormal basis $\{X,Y\}$ such that  $$\H[X,Y] = \rho \cdot X.$$
This means that for the following structure constants we have 
\begin{equation}\label{equation-1}
\lambda_{XY}^X=g([X,Y],X)=\rho=-g([Y,X],X)=-\lambda_{YX}^X
\end{equation}
and
\begin{equation}\label{equation-2}
\lambda_{XY}^Y=g([X,Y],Y)=0=-g([Y,X],Y)=-\lambda_{YX}^Y.
\end{equation}

Since $\F$ is Riemannian, we know that for each $V \in \V$, the adjoint action on the horizontal space $\H$ for any $E,F\in\H$ satifies
$$g([V,E],F)+g(E,[V,F])=0.$$
Writing this in terms of the structure constants we get 
\begin{equation}\label{equation-2}
\lambda_{V X}^X = \lambda_{V Y}^Y=0\ \ \text{and}\ \ \lambda_{V X}^Y = - \lambda_{V Y}^X.
\end{equation}
If we now substitute $e_i=X$ and $e_j=Y$ into equation (\ref{general-sectional-sg}) and simplify using the relations (\ref{equation-1}) and (\ref{equation-2}) we then yield
\begin{eqnarray*}
K(X,Y)
&=& -\rho^2-\tfrac{3}{4}\sum_{k=1}^n(\theta^k )^2.
\end{eqnarray*}
It then immediately follows from O'Neill's formula (\ref{equation-ONeill}) that  the leaf space $L$ has constant non-positive Gaussian curvature $K_{L} = - \rho^2$.


\appendix

\section{Sectional curvatures in Riemannian Lie groups}
\label{appendix}


Let $(G,g)$ be an $m$-dimensional Riemannian Lie group and $\{e_1,e_2,\dots,e_m\}$ be an orthonomal basis for its Lie algebra $\g$ satisfying the bracket relations 
$$[e_i,e_j] = \sum_k \lambda_{ij}^k e_k .$$
According to the Koszul formula for left-invariant vector fields on $G$, we then have 
\begin{eqnarray*} 
	g(\nabla_{e_i} {e_j}, {e_k})
	&=&\tfrac{1}{2} \{ g([{e_k},{e_i}],{e_j}) + g([{e_k},{e_j}],{e_i})+g({e_k},[e_{i},e_{j}])\}\\
	&=&\tfrac{1}{2}\sum_k( \lambda_{ki}^j +\lambda_{kj}^i+\lambda_{ij}^k ).
\end{eqnarray*} 
This implies the following for the Levi-Civita connection in terms of the structure constants
\begin{equation}\label{Levi-Civita}
	\nabla_{e_i}e_j = \sum_k g(\nabla_{e_i}e_j, e_k) \cdot e_k = \tfrac{1}{2}\sum_k ( \lambda_{ki}^j +\lambda_{kj}^i+\lambda_{ij}^k )\cdot e_k.
\end{equation}
In particular, we yield 
$$ \nabla_{e_j}e_j 
= \tfrac{1}{2}\sum_k ( \lambda_{kj}^j +\lambda_{kj}^j+\lambda_{jj}^k )\cdot e_k
=\sum_k \lambda_{kj}^j\cdot e_k$$
and hence
\begin{eqnarray*}
	g(\nabla_{e_i} \nabla_{e_j}e_j,e_i)
	&=&\sum_k \lambda_{kj}^j\cdot g(\nabla_{e_i} e_k,e_i)\\
	&=&\tfrac 12\sum_k \lambda_{kj}^j\cdot (\lambda_{ii}^k+\lambda_{ik}^i+\lambda_{ik}^i)\\
	&=&\sum_k \lambda_{kj}^j\cdot \lambda_{ik}^i.
\end{eqnarray*}
Now employing equation (\ref{Levi-Civita}) we obtain
\begin{eqnarray*}
	g(\nabla_{e_j} \nabla_{e_i}e_j,e_j)
	&=&\tfrac{1}{2}\sum_k ( \lambda_{ki}^j +\lambda_{kj}^i+\lambda_{ij}^k )\cdot g(\nabla_{e_j}e_k,e_i)\\
	&=&\tfrac{1}{4}\sum_k ( \lambda_{ki}^j +\lambda_{kj}^i+\lambda_{ij}^k )\cdot (\lambda_{ij}^k+\lambda_{ik}^j+\lambda_{jk}^i).
\end{eqnarray*}
Furthermore
\begin{eqnarray*}
	g(\nabla_{[e_i,e_j]}e_j,e_i)
	&=&\sum_k\lambda_{ij}^k\cdot g(\nabla_{e_k}e_j,e_i)\\
	&=&\tfrac 12\sum_k\lambda_{ij}^k\cdot (\lambda_{ik}^j+\lambda_{ij}^k+\lambda_{kj}^i).
\end{eqnarray*}
Now combining these terms we have the sectional curvature

\begin{eqnarray}\label{general-sectional-2}
K(e_i,e_j)&=&g(R(e_i,e_j)e_j,e_i)\\ \notag
&=& -\sum_k \lambda_{kj}^j \lambda_{ki}^i
	-\tfrac{1}{2}\cdot\sum_k\lambda_{ij}^k (\lambda_{ij}^k+ \lambda_{ik}^j +\lambda_{kj}^i) \\ 
&& +\tfrac{1}{4}\sum_k (\lambda_{ij}^k + \lambda_{ki}^j + \lambda_{kj}^i) \cdot ( \lambda_{ji}^k +\lambda_{ki}^j+\lambda_{kj}^i).\notag
\end{eqnarray}


\end{document}